\documentclass[a4paper,11pt]{amsart}
\usepackage[latin1]{inputenc}
\usepackage{graphicx}
\usepackage{amsmath,amssymb}

\newfont{\bb}{msbm10 at 11pt}
\newfont{\bbsmall}{msbm8 at 8pt}

\def\rth{\mathbb{R}^3}
\def\R{\mathbb{R}}

\def\N{\mathbb{N}}

\def\C{\mathbb{C}}

\def\H{\mathbb{H}}
\def\M{\mathbb{M}}
\def\D{\mathbb{D}}
\def\E{\mathbb{E}}

\def\esf{\mathbb{S}}

\def\l{{\lambda}}

\newcommand{\al}{{\alpha}}
\newcommand{\te}{{\theta}}

\newcommand{\g}{{\gamma}}

\newcommand{\ve}{{\varepsilon}}

\newcommand{\Jac}{\mbox{\rm Jac}}

\def\supp{\mathop{\rm supp}}

\newtheorem{theorem}{Theorem}[section]
\newtheorem{lemma}[theorem]{Lemma}
\newtheorem{proposition}[theorem]{Proposition}

\newtheorem{remark}[theorem]{Remark}
\newtheorem{corollary}[theorem]{Corollary}

\date{April 22, 2010}

\begin{document}

\title{Parabolic stable surfaces with constant mean curvature}
\author{Jos\'e M. Manzano, \ Joaqu\'\i n P\'erez \ and \ M. Magdalena Rodr\'\i guez}
\thanks{Research partially supported by a Spanish MEC-FEDER Grant no. MTM2007-61775 and a Regional J. Andaluc\'\i a Grant no. P06-FQM-01642.\newline {\it Addresses:} Jos\'e M. Manzano ({\tt jmmanzano@ugr.es}), Joaqu\'\i n P\'erez ({\tt jperez@ugr.es}), both at the Department of Geometry and Topology, Universidad de Granada (SPAIN), M. Magdalena Rodr\'\i guez ({\tt magdalena@mat.ucm.es}) at the Department of Algebra, Universidad Complutense de Madrid (SPAIN).}

\maketitle
\begin{abstract}
We prove that if $u$ is a bounded smooth function in the kernel of a
nonnegative Schr\"{o}dinger operator $-L=-(\Delta +q)$ on a
parabolic Riemannian manifold $M$, then $u$ is either identically
zero or it has no zeros on $M$, and the linear space of such
functions is 1-dimensional. We obtain consequences for  orientable,
complete stable surfaces with constant mean curvature $H\in \R $ in
homogeneous spaces $\E (\kappa ,\tau )$ with four dimensional
isometry group. For instance, if $M$ is an orientable, parabolic,
complete immersed surface with constant mean curvature $H$ in
$\H^2\times \R $, then $|H|\leq \frac{1}{2}$ and if equality holds, then $M$
is either an entire graph or a vertical horocylinder.
\vspace{.3cm}


\end{abstract}

\section{Introduction}
\label{intro}
\label{sec1}

Constant mean curvature $H\in \R $ surfaces (briefly, $H$-surfaces) in
a Riemannian three-manifold $N$ constitute a natural object
of study since they are critical points of the functional
\[
F=\mbox{Area }-2H\cdot \mbox{ Volume}
\]
for compactly supported normal variations (along the paper, we will
restrict ourselves to 2-sided surfaces in order to insure that the
variational vector field of such a variation is essentially given by
a function instead of a section of the normal bundle of the
surface).

Among complete $H$-surfaces in $N^3$, the subclass of {\it
  stable} $H$-surfaces is perhaps the first one to be understood and
completely described; one of the main reasons for this is that, under
some natural conditions, limits of complete $H$-surfaces produce
complete stable $H$-surfaces (see for instance the recent Stable Limit
Leaf theorem by Meeks, P\'erez and Ros~\cite{mpr18}).  Here, stability
for an $H$-surface $M$ means that the second derivative of $F$ is
nonnegative for all compactly supported normal variations of $M$ in
$N^3$ (we remark that this is a stronger notion than the
usual stability associated to isoperimetry, where only normal
variations which preserve volume up to first order are considered). It
turns out that such a second derivative is explicitly given by an
expression of the type
\[
{\mathcal Q}(f,f)=-\int _MfLf,
\]
where $f$ is the normal part of the variational vector field of the
variation and $L$ is a Schr\"{o}dinger type operator on $M$, namely
$L=\Delta +q$ where $\Delta$ is the Laplacian with respect to the
induced metric on $M$ and $q$ is certain smooth function on $M$.
This extrinsic formulation justifies the interest of the study of
purely intrinsic operators $L=\Delta +q$ on an abstract Riemannian
surface $(M,ds^2)$ where $q\in C^{\infty }(M)$, under the condition
that $-\int _MfLf\geq 0$ for all $f\in C^{\infty }_0(M)$ (we will
brief this condition by writing $-L\geqslant 0$).

We will start by imposing a global condition on the metric $ds^2$ of
conformal nature, namely to be {\it parabolic}\footnote{A Riemannnian
  manifold $M$ is {\it parabolic} if every positive superharmonic
  function on $M$ must be constant. We remark that
  % except in Section~\ref{seccptbdry} below,
all manifolds are assumed to have empty boundary.}, to conclude that
some properties of operators $L=\Delta +q$ with $-L\geqslant 0$ on a
parabolic Riemannian manifold mimic the ones of the first
eigenfunctions of a clasical Schr\"{o}dinger operator defined on a
compact subdomain. For instance, if there exists a non identically
zero {\it bounded} solution of $Lu=0$ on $M$, then $u$ vanishes
nowhere and the linear space of such functions is 1-dimensional (see
Theorem~\ref{thm1} and Corollary~\ref{cor:qag} for slightly more
general statements of this property).

Next we will apply the above intrinsic result to obtain conditions
under which a complete stable $H$-surface in a simply-connected
homogeneous 3-manifold $\E (\kappa ,\tau )$ with four dimensional
isometry group is a (vertical) multigraph or even an entire graph.
The spaces $\E (\kappa ,\tau )$ appear naturally in the
classification of the Thurston geometries, together with the space
forms $\R^3,\esf ^3,\H ^3$ (with 6-dimensional isometry group) and
the solvable Lie group Sol$_3$ (3-dimensional isometry group). $\E
(\kappa ,\tau )$ gives a common framework to the product spaces
$\esf^2\times \R $, $\H^2\times \R $ and the nontrivial Riemannian
fibrations given by the Heisenberg group Nil$_3$ (over $\R^2$), the
Berger spheres (over $\esf^2$) and the universal cover
$\widetilde{\mbox{PSL}_2}(\R )$ of the unit tangent bundle of the
hyperbolic plane (over $\H^2)$.  The geometry of $H$-surfaces in
these spaces $\E (\kappa ,\tau )$ is a field of intensive research
nowadays, especially since Abresch and Rosenberg~\cite{AbRo1}
described a Hopf-type holomorphic quadratic differential on any such
surface. As an application of the results obtained in the abstract
setting, we prove the following results:
\begin{enumerate}
\item There are no orientable, complete, immersed, parabolic, stable
  $H$-surfaces $M$ in $\H^2\times\R$ for any value of $H>\frac{1}{2}$,
  and if $H=\frac{1}{2}$ then either $M$ is a vertical horocylinder or
  an entire vertical graph (Theorem~\ref{th:localgraph} and
  Corollary~\ref{th:planes}).  If we drop
  the hypothesis of parabolicity, we obtain nonexistence of
  orientable, complete, immersed stable $H$-surfaces $M$ in
  $\H^2\times\R$ for values $H>\frac1{\sqrt{3}}-\ve $ for some $\ve
  >0$ (the case $H>\frac{1}{\sqrt{3}}$ had been proven by Nelli and
  Rosenberg~\cite{ner1}, see also Remark~9.11 in~\cite{mpr19} where
  this nonexistence property is stated for $H>\frac1{\sqrt{3}}-\ve $;
  we also remark that Salavessa obtained a related inequality for $|H|$
  valid for certain global graphs with parallel mean curvature in product
 manifolds and in calibrated manifolds~\cite{sali1,sal1}).

\item There are no orientable, complete, immersed, parabolic, stable
  $H$-surfaces $M$ in Nil$_3$ for any $H>0$, and if
$H=0$, then either $M$ is a vertical plane or an entire vertical
graph  (Theorem~\ref{th:localgraph} and
  Corollary~\ref{th:planes}). This
  result is sharp, since every solution of the Bernstein problem in
  Nil$_3$ being conformally $\C $ provides an example of the last ones
  (Fern\'andez and Mira~\cite{fm1} constructed a 2-parameter family of
  such solutions for every nonzero quadratic differential on $\C $).
  By the Daniel sister correspondence~\cite{da1}, the result in the
  first sentence of item 1 above is also sharp. The second and third
  sentences of item 1 can be also adapted to Nil$_3$, see
  Corollary~\ref{th:mejora}.

\item There are no orientable, complete, immersed, parabolic, stable
  $H$-surfaces $M$ in any Berger sphere (it is expected that this
  result holds without assuming parabolicity, see Meeks, P\'erez and
  Ros~\cite{mpr19} where it is proved assuming either compactness of
  $M$ or nonnegative scalar curvature of the ambient Berger sphere).

\item We start the study of the Bernstein problem for {\it horizontal}
  minimal graphs in Nil$_3$. A horizontal graph over a domain $\Omega $
  of the $(y,z)$-plane is the set $\{ (0,y,z)*(u(y,z),0,0) \mid (y,z)\in \Omega \} $
  where $u$ is a $C^2$-function defined in $\Omega $ (here we are using the natural Lie
  group multiplication $*$ of Nil$_3$, see Section~\ref{sechorizgraphNil}).
  Recall that a usual {\it vertical} graph over a domain $\Omega $ of the
  $(x,y)$-plane in Nil$_3$ is the set $\{ (x,y,0)*(0,0,u(x,y))\mid
  (x,y)\in \Omega \} $, where $u(x,y)$ is a $C^2$-function defined on $\Omega$;
  a direct comparison justifies the use of the
  word {\it horizontal} for the graphs in our study.
  In this setting, we prove that a parabolic {\it complete} horizontal minimal graph
  $M$ in Nil$_3$ must be a vertical plane (see Theorem~\ref{thm4.1} for a more general version
  of this result for horizontal multigraphs). We remark that every {\it entire} horizontal graph
  (i.e. $\Omega $ equals the $(y,z)$-plane) is proper, hence complete.
\end{enumerate}
The authors would like to thank J. M. Espinar, L. Mazet, W. H. Meeks and A. Ros for some helpful conversations.

\section{The operator $\Delta+q$.}
\label{secabstract}
Along this section, $(M,ds^2)$ will denote a connected (not necessarily
complete) Riemannian $n$-manifold without boundary. Consider the
operator in $M$ given by
\[
L=\Delta+q,
\]
where $\Delta$ stands for the Laplacian with respect to the metric
$ds^2$ acting on functions and $q\in C^\infty(M)$.

We next state the main result of this section.
\begin{theorem}
\label{thm1}
Let $(M^n,ds^2)$ be a Riemannian parabolic manifold.
Consider an operator ${L=\Delta +q}$, where $q\in C^{\infty }(M)$. Let
$u,v\in C^{\infty }(M)$ such that $u$ is bounded, $v>0$ and $uvLu\geq u^2Lv$
on~$M$. Then, $u/v$ is constant.
\end{theorem}
Before proving Theorem~\ref{thm1}, we will make some comments.

\begin{enumerate}
\item Since $L=\Delta +q$, the hypothesis $uvLu\geq u^2Lv$ is obviously equivalent to
 $uv\Delta u\geq u^2\Delta v$. In other words, we can restrict to the particular case
 $q=0$.
\item We have stated the theorem under the hypothesis $(M,ds^2)$ is parabolic, although
the equivalent condition (see e.g. Proposition~4.1 in~\cite{Troy1}) used in its proof will be
\begin{description}
\item[$(\star )$]
There exists
a sequence of cut-off functions $\{ \varphi _j\} _j\subset C^{\infty }_0(M)$ such that
$0\leq \varphi _j\leq 1$ in $M$, the compact sets $\varphi _j^{-1}(1)$ form an
increasing exhaustion of $M$, and the sequence of energies $\{ \int _M|\nabla \varphi _j|^2\} _j$
tends to zero as $j\to \infty $.
\end{description}
\end{enumerate}
\begin{proof}
 The argument is inspired in ideas of Berestycki, Caffarelli and
  Nirenberg~\cite[Section 4]{BCN} (also
  see Ambrosio and Cabr\'e~\cite[Proposition~2.1]{amca1}). As said
  before this proof, we can assume $q=0$. First note that $u/v$
  is smooth on $M$ and its gradient is $\nabla (u/v)=v^{-2}(v\nabla
  u-u\nabla v)$. Hence, the divergence of the smooth vector field
  $v^2\nabla (u/v)$ is div$\left( v^2\nabla (u/v)\right) = v\Delta
  u-u\Delta v$, from where we obtain
\begin{equation}
\label{eq:thm1A}
u\, \mbox{div}\left( v^2\nabla (u/v)\right) =uvLu-u^2Lv\geq 0.
\end{equation}
Consider a sequence $\{ \varphi _j\} _j\subset C^{\infty }_0(M)$
satisfying the property $(\star )$. Given $j\in \N$, the standard
argument in the sense of distributions applied to the function
$\varphi _j^2\, u/v\in H_0^1(M)$ (as usual, $H_0^1(M)$ denotes the
closure with respect to the standard Sobolev norm of the linear space
$C_0^{\infty }(M)$ of compactly supported smooth functions on $M$)
 and to the smooth vector field
$v^2\nabla (u/v)$ yields
\[
\begin{array}{rcl}
  0&=&
  {\displaystyle \int _M
    \langle \nabla (\varphi _j^2u/v) ,v^2\nabla
    (u/v)\rangle
    +\int _M\frac{\varphi _j^2u}{v}\mbox{ div}\left (v^2\nabla
      (u/v)\right) }\\
  &=&
  {\displaystyle 2\int _M\varphi _juv \langle \nabla \varphi _j
    ,\nabla (u/v) \rangle
    +\int _M\varphi_j^2v^2\left| \nabla
      (u/v)\right| ^2
    +\int _M\frac{\varphi _j^2u}{v}\mbox{ div}\left (v^2\nabla
      (u/v)\right) }\\
\end{array}
\]
where $\langle \cdot ,\cdot \rangle $ stands for the Riemannian metric
on $M$. We define $\Omega _j=\varphi _j^{-1}(1)$.
The last equality together with (\ref{eq:thm1A}) lead to
\begin{eqnarray}
  {\displaystyle \int _M\varphi _j^2v^2|\nabla
    (u/v)|^2}&\leq &{\displaystyle -2\int _M\varphi _juv \langle \nabla
    \varphi _j,\nabla (u/v) \rangle }\nonumber
  \\
  &=&{\displaystyle -2\int _{\supp (\varphi _j)-\Omega _j}\varphi _juv \langle \nabla
    \varphi _j,\nabla (u/v) \rangle }\nonumber
  \\
  &\leq &{\displaystyle 2\left| \int _{\supp (\varphi _j)-\Omega _j}
  \varphi _juv \langle
      \nabla \varphi _j,\nabla (u/v) \rangle \right| }\nonumber
  \\
  &\leq & {\displaystyle 2\left( \int _{\supp (\varphi _j)-\Omega _j}\varphi _j^2v^2
      |\nabla (u/v)|^2\right) ^{1/2} \left( \int
      _{\supp (\varphi _j)-\Omega _j}u^2|\nabla \varphi _j|^2\right) ^{1/2} },
\label{eq:thm1B}
\end{eqnarray}
where we have used the Schwarz inequality. Since $u$ is bounded and
the functions $\varphi _j$ satisfy $(\star )$, we have
\begin{equation}
\label{eq:thm1Ba}
\int _{\supp (\varphi _j)-\Omega _j}u^2|\nabla \varphi _j|^2\leq C\, E(\varphi _j)
\end{equation}
where $C>0$ is independent of $j$
and the energy $E(\varphi _j)=\int _M|\nabla \varphi _j|^2$ tends to zero as $j\to \infty $.
Plugging this
information in (\ref{eq:thm1B}) and splitting its
left-hand-side in two integrals over $\Omega _j$ and
$\supp (\varphi _j)-\Omega _j$, we have
\begin{equation}
\label{eq:thm1C}
\int _{\Omega _j}v^2|\nabla (u/v)|^2+
\int _{\supp (\varphi _j)-\Omega _j}\varphi _j^2v^2|\nabla (u/v)|^2
\leq
2\sqrt{C\, E(\varphi _j)}\left( \int _{\supp (\varphi _j)-\Omega _j}\varphi _j^2v^2 |\nabla (u/v)|^2\right) ^{1/2}.
\end{equation}
Estimating by below the first summand in the last left-hand-side by
zero and simplifying we have
\begin{equation}
\label{eq:thm1D}
\int _{\supp (\varphi _j)-\Omega _j}\varphi _j^2v^2|\nabla (u/v)|^2
\leq 4C\ E(\varphi _j).
\end{equation}
Estimating by below the second summand in (\ref{eq:thm1C}) by zero and
using (\ref{eq:thm1D}),
\[
%\begin{equation}
%\label{eq:thm1E}
\int _{\Omega _j}v^2|\nabla (u/v)|^2
\leq
2\sqrt{C\, E(\varphi _j)}\left( \int _{\supp (\varphi _j)-\Omega _j}\varphi _j^2v^2 |\nabla (u/v)|^2\right) ^{1/2}\leq 4C\, E(\varphi _j),
%\end{equation}
\]
for all $j\in \N$. Taking $j\to \infty $ we infer that
%\[
%\int _Mv^2|\nabla (u/v)|^2 \leq 4C,
%\]
%which in turn implies
%\begin{equation}
%\label{eq:thm1F}
%\lim _{j\to \infty }\int _{\supp (\varphi _j)-\Omega _j}\varphi _j^2v^2|\nabla (u/v)|^2=0.
%\end{equation}
%Finally, (\ref{eq:thm1C}) and (\ref{eq:thm1F}) imply
%\[
%\begin{array}{rcl}
%  {\displaystyle \int _{\Omega _j}v^2|\nabla (u/v)|^2}
%  &=&
%  {\displaystyle
%    \int _{\Omega _j}\varphi _j^2v^2|\nabla (u/v)|^2}
%  \\
%  &\leq &
%  {\displaystyle
%    \int _{M}\varphi _j^2v^2|\nabla (u/v)|^2}
%  \\
%  &\leq &
%  {\displaystyle
%    2\sqrt{C}\left( \int _{\supp (\varphi _j)-\Omega _j}\varphi _j^2v^2|\nabla
%      (u/v)|^2\right)^{1/2}
%    \stackrel{(j\to \infty )}{\longrightarrow }0,}
%\end{array}
%\]
%and thus,
$\int _{M}v^2|\nabla (u/v)|^2=0$.
Since $v>0$ in $M$, we conclude that $u/v$ is constant, as desired.
\end{proof}

We continue with an operator $L=\Delta +q$ on $(M^n,ds^2)$, $q\in
C^{\infty }(M)$. As usual, we associate to $L$ the quadratic form
\begin{equation}
\label{eq:indexform}
{\mathcal Q}(f,f)=- \int_M f\, Lf= \int_M\left(|\nabla f|^2-qf^2\right)
,\quad f\in C_0^\infty(M),
\end{equation}
which can be continuously extended to the Sobolev space $H_0^1(M)$.
We say that $-L$ is {\it nonnegative} on $M$ if ${\mathcal Q}(f,f)\geq
0$ for all $f\in C_0^{\infty }(M)$, and we write $-L\geqslant 0$ in
this case. Functions in the nullity of ${\mathcal Q}$ are called {\it
Jacobi
  functions.} By standard elliptic theory, Jacobi functions are always
smooth and lie in the kernel of $L$.  Fischer-Colbrie~\cite{fi1}
proved the following useful characterization of nonnegativity of $-L$
in terms of existence of positive Jacobi functions (see also Lemma 2.1
in Meeks, P\'erez and Ros~\cite{mpr19}); note that both~\cite{fi1} and~\cite{mpr19} are
stated and proved for surfaces, but the proof can be directly extended
to the $n$-dimensional case):

\begin{lemma}
  \label{lemaFC}
  In the above situation, the following statements are equivalent:
  \begin{enumerate}
  \item $-L\geqslant    0$ on $M$.
  \item There exists a positive Jacobi function on $M$.
  \item There exists a positive function $u\in C^\infty(M)$ such
    that $Lu\leq 0$.
  \end{enumerate}
\end{lemma}

%Suppose $(M,ds^2)$ is complete.  Given $p_0\in M$, we denote by
%$d(\cdot ,p_0)$ the Riemannian distance on $M$ to $p_0$ and by
%$B_M(p_0,r)=\{ p \in M \mid d(p,p_0)<r\} $ the metric ball in $M$ of
%radius $r$ centered at $p_0$. We say that $M$ has {\it at most
%  quadratic area growth} (QAG) when the function $r>0\mapsto
%r^{-2}\mbox{Area}(B_M(p_0,r))$ is bounded. This condition is
%independent of the point $p_0\in M$.

Suppose $(M^n,ds^2)$ is complete.  Given $p_0\in M$, we denote by
$d(\cdot ,p_0)$ the Riemannian distance on $M$ to $p_0$ and by
$B(p_0,r)=\{ p \in M \mid d(p,p_0)<r\} $ the metric ball in $M$ of
radius $r$ centered at $p_0$. We say that $M$ has {\it at most
  quadratic volume growth} when the function $r>0\mapsto
r^{-2}\mbox{Vol}(B(p_0,r))$ is bounded. This condition is
independent of the point $p_0\in M$. It is well-known that
if $M$ has at most quadratic volume growth, then it is parabolic
(Cheng-Yau~\cite{chya1}, see also Corollary~7.4 of~\cite{gri1}).
In Theorem~2.11 of~\cite{mpr19},
Meeks, P\'erez and Ros proved that if
$M$ is a simply-connected surface with quadratic area growth (although
only parabolicity is required in their argument) and
$-L\geqslant 0$, then a bounded
solution of $Lu=0$ does not change sign on $M$.  The following
corollary is a slight generalization of that result.

\begin{corollary}
 \label{cor:qag}
  Let $(M^n,ds^2)$ be a parabolic Riemannian manifold. Suppose
  $L=\Delta +q$, $q\in C^{\infty }(M)$, satisfies $-L\geqslant 0$ on
  $M$. If $u\in C^{\infty }(M)$ is a bounded function such that
  $uLu\geq 0$, then $Lu=0$ and either $u=0$  or $u$ never vanishes.
\end{corollary}
\begin{proof}
  As $-L\geqslant 0$ on $M$, Lemma~\ref{lemaFC} insures that there
  exists $v\in C^{\infty }(M)$ such that $v>0$ and $Lv=0$ on $M$. By
  Theorem~\ref{thm1}, $u$ must be a multiple of $v$ from where the
  corollary follows directly.
\end{proof}

\medskip

Next we show a direct application of Corollary~\ref{cor:qag}.
Consider functions $f\in C^{\infty }(\R )$ and $u\in C^{\infty }(M)$
such that $Lu=0$, and define $v=f(u)\in C^{\infty }(M)$. Note that
\begin{equation}\label{eq:f(u)}
Lv=f''(u)|\nabla u|^2+q\left(f(u)-f'(u) u\right).
\end{equation}

\begin{lemma}\label{lem:qpos}
Suppose that the hypotheses of Corollary~\ref{cor:qag}
 hold.
\begin{enumerate}
\item If $q\geq 0$, then $q=0$. In particular, any Jacobi function bounded
  above or below on $M$ must be constant.
\item If $q\leq 0$, then any bounded Jacobi function $u$ on $M$ is
  constant, and either $u=0$ or $q=0$ in $M$.
\end{enumerate}
\end{lemma}
\begin{proof}
We first prove 1. By Lemma~\ref{lemaFC}, there exists a
positive solution $w$ of $Lw=0$.  Thus $v=e^{-w}$ is smooth, positive and bounded on $M$.
  Furthermore, (\ref{eq:f(u)}) implies that it satisfies
  $Lv=v\left(|\nabla w|^2+q(1+w)\right)$. Since $q\geq 0$ on $M$ and
  $w$ is positive, then $Lv\geq 0$ on $M$.  Therefore,
  Corollary~\ref{cor:qag} implies $Lv=0$, and so
$|\nabla w|=0$  and $q=0$. The second sentence in item~1 follows directly from the
parabolicity of $M$, since $L=\Delta $.
%Note that from $\Delta u=Lu=0$ we cannot deduce directly that $u$ is constant
%since $M$ is not necessarily parabolic. Instead, we argue as follows. If $u> R$
%  (resp. $u<R$) for some constant $R$, we consider $v=e^{R-u}$
%  (resp. $v=e^{u-R}$).  Then $0<v<1$ and $ Lv= v |\nabla u|^2
%\geq 0$. By Corollary~\ref{cor:qag}, $Lv=0$, i.e. $|\nabla u|=0$ and 1 is proved.

To show 2, apply (\ref{eq:f(u)}) to the bounded smooth function $v=u^2$ we have $Lv=2|\nabla u|^2-qu^2\geq 0$.
  From Corollary~\ref{cor:qag} we get $Lv=0$, hence $|\nabla u|=0$
  and $q u^2=0$ and we are done.
\end{proof}

\section{Stable $H$-surfaces in 3-manifolds.}
\label{secstable} Let $M$ be a two-sided immersed surface of
constant mean curvature $H\in \R $ (briefly, an $H$-surface) in a
Riemannian 3-manifold~$N$.  From a variational viewpoint, $M$ is a
critical point of the functional Area$-2H\cdot $ Volume. More
precisely, denote by $x\colon M\to N^3$ an isometric immersion with
constant mean curvature $H$ and globally defined unit normal vector
field $\eta $. Given a function $f\in C^{\infty }_0(M)$ and a
compact smooth domain $\Omega \subset M$ which contains the support
of $f$, we consider any variation of $x$ given by a differentiable
map $X\colon (-\ve ,\ve )\times M\to N^3$, $\ve
>0$, such that $X(0,\cdot ) = x$ on $M$, $X(t,p) = x(p)$ for all
$(t,p)\in (-\ve ,\ve )\times [M-\Omega ]$, and $\left.
\frac{\partial X}{\partial t}\right| _0=f\eta $. We associate to $X$
the area function Area$(t) =$ Area$(X(t,\cdot ))$ (note that for
small $t$, the map $X(t,\cdot )\colon M\to N^3$ is an immersion) and
the signed volume function Vol$(t)= \int _{[0,t]\times \Omega }\Jac
(X)\, dV$. Then, the first derivative at $t=0$ of Area$(t)-2H
$Vol$(t)$ vanishes. It is also well-known that the second variation
formula of Area$-2H \cdot $Volume at $t=0$ is given by (see e.g.
Barbosa, Do Carmo and Eschenburg~\cite{bce1})
\begin{equation}
\label{QHsurf}
{\mathcal Q}(f,f) = \left. \frac{d^2}{dt^2}\right| _{t=0}[\mbox{Area}(t) -
2H \mbox{Vol}(t)] = -\int _MfLf\, dA= \int _M ( |\nabla f|^2-qf^2)\, dA,
\end{equation}
where $L=\Delta +q$ is a Schr\"{o}dinger type operator acting on
smooth functions called the {\it stability operator} of $M$, $\Delta $
stands for the Laplacian with respect to the induced metric on $M$ and
$q$ is the smooth function given by
\[
q=|A|^2+\mbox{Ric}(\eta )
\]
(here $A$ denotes the shape operator of $M$).  The surface $M$ is
said to be {\it stable} when $-L$ is a nonnegative operator.  We
also say that $u\in C^\infty(M)$ is a {\it Jacobi function} when
$Lu=0$.

Note that if $X$ is a Killing vector field on $N^3$, then
$\langle X,\eta \rangle $ is a Jacobi function on $M$. Since every stable
$H$-surface admits a positive Jacobi function by Lemma~\ref{lemaFC}, then
Theorem~\ref{thm1} has the following direct consequences:
\begin{itemize}
\item If $M$ is a parabolic stable $H$-surface and $N^3$ admits a nonzero
Killing vector field which is bounded when restricted to $M$, then
either $M$ is
invariant under the 1-parameter group of isometries generated by $X$,
or the linear space of bounded Jacobi functions and
the cone of positive Jacobi functions on $M$ are both generated
by the same function. From here we deduce that if $M$ is a
parabolic stable $H$-surface in a product manifold $\Sigma \times \R $,
then either $M$ entirely vertical or it is transverse to the unit
vertical Killing vector field $\frac{\partial }{\partial t}$.
\item If $M$ is a parabolic stable $H$-surface and $N^3$ admits
two linearly independent Killing vector fields which are bounded
when restricted to $M$, then $M$ is invariant under the 1-parameter group of
isometries generated by some linear combination of $X_1,X_2$.
\item If $M$ is a parabolic stable $H$-surface and $N^3$ admits
three linearly independent Killing vector fields which are bounded
when restricted to $M$, then $M$ has constant Gaussian curvature.
\end{itemize}

In the sequel, we will study the special case in which $N^3$ is a
simply-connected homogeneous 3-manifold whose isometry group has
dimension 4. These homogeneous spaces are classified in terms of two
real numbers $\kappa,\tau$ with $\kappa\neq 4\tau^2$, and are usually
denoted by $\E(\kappa,\tau)$. The space $\E(\kappa,\tau)$ admits a
fibration $\pi $ over the complete simply-connected surface
$\M^2(\kappa)$ of constant curvature $\kappa$ (the sphere
$\esf^2(\kappa)$ when $\kappa>0$, the Euclidean plane $\R^2$ when
$\kappa=0$ and the hyperbolic plane $\H^2(\kappa)$ when $\kappa<0$);
here $\tau$ is the bundle curvature.  The fibers of $\pi \colon \E
(\kappa ,\tau )\to \M ^2(\kappa )$ are geodesics, and translations
along these fibers generate a unit Killing vector field $E_3$, called the
{\it vertical vector field.}

If $\tau=0$, then $\E(\kappa,\tau)$ is the product space
$\M^2(\kappa)\times\R$.  In the case $\tau \neq 0$, we get three types
of manifolds depending on the sign of $\kappa $: if $\kappa >0$ we
have the {\it Berger spheres,} if $\kappa =0$ we obtain the {\it
  Heisenberg space} Nil$_3$, and the case $\kappa <0$ corresponds to
the universal covering of PSL$_2(\R )$, the unit tangent bundle of
$\H^2 $.  The above description could be
extended for $\kappa= 4\tau^2$, obtaining the Euclidean space $\R^3$
(when $\kappa=\tau=0$) or a round 3-sphere $\esf^3$ (when
$\kappa=4\tau^2\neq 0$), although these space forms have isometry
group of dimension~6 and we will not consider them.

A consequence of an intrinsic estimate of the distance to the boundary
of an immersed stable $H$-surface due to Rosenberg~\cite{rose4}  is
the nonexistence of
immersed, stable $H$-surfaces in $\E(\kappa,\tau)$ provided that
$H^2>\frac{\tau^2-\kappa} 3$, other than $\esf^2(\kappa)\times\{0\}$
in $\esf^2(\kappa)\times\R$.
It is expected that the right condition for such nonexistence result
is $H^2>\frac{-\kappa} 4$.  Our next result achieves such bound under
the additional assumption of parabolicity.  Also see
Corollary~\ref{th:mejora} below for a slight improvement of the
inequality $H^2>\frac{\tau ^2-\kappa }{3}$ when the hypothesis of
parabolicity is removed.

Given an immersed $H$-surface $M\looparrowright\E(\kappa,\tau)$ with
unit normal vector field $\eta $, the bounded Jacobi function
$\eta _3=\langle \eta ,E_3\rangle $ is called the {\it angle function} of $M$.
The surface $M$  is called a {\it vertical multigraph} if $\eta _3$ does
not vanish (i.e. it is transverse to the fibration $\pi
\colon \E (\kappa ,\tau )\to \M ^2 (\kappa )$), and is said to be
a {\it cylinder} over a curve $\gamma \subset \M ^2(\kappa )$ if $\eta _3=0$
(i.e. $M=\pi ^{-1}(\g )$).

\begin{theorem}
\label{th:localgraph}
  Let $M$ be an orientable, parabolic, complete, immersed stable $H$-surface
  in $\E(\kappa,\tau)$.  Then, one of the following statements hold:
  \begin{enumerate}
  \item $\E(\kappa,\tau)=\esf^2(\kappa)\times\R$, $H=0$ and $M$ is a slice
    $\esf^2(\kappa)\times\{t\}$ for some $t\in\R$.
    \item $H^2\leq \frac{-\kappa} 4$ and $M$ is either a vertical
    multigraph or a vertical cylinder over a complete curve of
    geodesic curvature $2H$ in $\M^2(\kappa)$.
  \end{enumerate}
\end{theorem}
\begin{proof}
  Since $\eta _3$ is a bounded Jacobi function and
  $M$ is parabolic, then Corollary~\ref{cor:qag} assures that either $\eta _3$
  is identically zero or it never vanishes. In the first case, $M$ is
  a vertical cylinder over a complete curve of geodesic curvature $2H$
  by Lemma~\ref{lem:cylinders} in the appendix. Since $M$ is stable,
  we conclude from Proposition~\ref{prop:cyl} in the same appendix
  that $\kappa\leq -4H^2$, which finishes this case.

  Assume from now on that $\eta _3$ has no zeros on $M$.  Then $M$ is a
  vertical multigraph and we can suppose $\eta _3>0$. Assume that
  $H^2>\frac{-\kappa}{4}$ and we will prove that item 1 holds.

If there exists a constant $c>0$ such that $\eta _3\geq c$, then the
  vertical projection $\pi $ restricts to $M$ as a covering map
over the simply-connected surface $\M^2(\kappa )$. Hence $M$ is an entire vertical
  graph (i.e. $M$ intersects exactly once every fiber of $\pi $ or
  equivalently, $M$ is the image of a global section of the Riemmanian
  fibration $\pi \colon \E (\kappa ,\tau )\to \M ^2(\kappa )$). Now
  assume $\E (\kappa ,\tau )$ is not a Berger sphere and
  consider a sphere $S_H\subset \E (\kappa ,\tau )$ of constant mean
  curvature $H$, which exists since $\frac{-\kappa} 4<H^2$ (see
  Remark~\ref{rem:spheres} below).  By applying the maximum principle
  to $M$ and $S_H$ we conclude that $M=S_H$, which can be easily seen
  to be possible only if item 1 of the theorem holds. In the case
  $\E (\kappa ,\tau )$ is a Berger sphere the above argument might
  fail, since we need to start with a sphere $S_H$ {\it disjoint} from
  $M$, which does not necessarily exist. In this case, we simply
  observe that $M$ is compact since it is a global graph over
  $\esf^2(\kappa )$, which contradicts item~1 of Corollary~9.6
  in~\cite{mpr19}.

  Therefore, there exists a sequence of points $\{p_n\}_n$ in $M$ such
  that $\eta _3(p_n)\to 0$ as $n\to +\infty$. Consider the isometry
  $\phi_n$ of $\E(\kappa,\tau)$ obtained as composition of the
  vertical translation $T_n$ which maps $p_n$ to $\pi (p_n)$ (here we
  identify $\pi (p_n)$ with its related point in $\E (\kappa ,\tau )$
  at height zero with respect to the fibration $\pi $) with the
  horizontal translation which maps $T_n(p_n)$ to a previously chosen
  point ${\bf 0}$ of $\E(\kappa,\tau)$, to which we will call the
  origin (here, the word ``horizontal'' refers to the legendrian lift
  of the isometry of $\M^2(\kappa )$ through $\pi $, a lift which
  preserves the natural distribution orthogonal to the vertical
  fibers; for instance, in Section~\ref{sechorizgraphNil} below, this
  distribution is generated by the vector fields $E_1,E_2$).  Then
  $\phi_n(p_n)={\bf 0}$ and the surface $M_n=\phi_n(M)$ is a complete
  stable $H$-surface which contains ${\bf 0}$. Let $\eta _{n,3}$ be the
  angle function of $M_n$ (which is nothing but $\eta _3\circ \phi
  _n^{-1}$). Then $\eta _{n,3}>0$ on $M_n$ and $\eta _{n,3}({\bf 0})\to 0$ as
  $n\to+\infty$.

  Since $\{ M_n\} _n$ is a sequence of $H$-surfaces with uniformly
  bounded second fundamental form (this follows from curvature
  estimates for stable constant mean curvature surfaces, see
  Schoen~\cite{sc3}) and an accumulation point at the origin, standard
  convergence theorems (see for instance Theorem~4.2.2 in P\'erez and
  Ros~\cite{pro2} for the minimal case, which can be extended to the
  case of fixed constant mean curvature with minor changes) give that
  after extracting a subsequence, there exist neighborhoods $U_n$ of
  ${\bf 0}$ in $M_n$ which converge uniformly in the $C^k$-topology
  for every $k$ to a (not necessarily complete) $H$-surface $U_{\infty
  }$ which contains the origin ${\bf 0}$. Clearly, $U_{\infty }$ is
  stable and its angle function $\eta _{3,\infty }= \lim _n\eta _{3,n}|_{U_n}$
  satisfies $\eta _{3,\infty }\geq 0$ in $U_{\infty }$ and $\eta _{3,\infty
  }({\bf 0})=0$. Applying the maximum principle to $\eta _{3,\infty }$
  (see e.g. Assertion 2.2 in~\cite{mpr19}) we conclude that
  $\eta _{\infty,3}$ is identically zero in $U_{\infty }$.  Therefore,
  $U_\infty$ is contained in a vertical cylinder $C_{\g }$ over a
  curve $\g \subset \M ^2(\kappa )$ with constant geodesic curvature
  $2H$. Since the surfaces $M_n$ are complete without boundary and
  have uniformly bounded fundamental form, an analytic prolongation
  argument insures that the maximal sheet which contains $U_{\infty }$
  in the accumulation set of the $\{ M_n\} _n$ is the whole cylinder
  $C_{\g }$ (after extracting a subsequence). In particular $C_{\g }$
  is stable, which is a contradiction since as we said above, there
  are no stable cylinders of constant mean curvature $H$ with
  $H^2>\frac{-\kappa} 4$. Now the proof is complete.
\end{proof}

\begin{remark}\label{rem:spheres}
  {\rm In the above proof, we have used the existence of (immersed)
    $H$-spheres in $\E (\kappa ,\tau )$ for all values of $H$ with
    $H^2>\frac{-\kappa} 4$. We now give a brief explanation of this
    well-know result.  In the product space $\M^2(\widetilde
    \kappa)\times\R$, there are (embedded) rotational spheres of
    constant mean curvature $\widetilde H$ whenever $4\widetilde
    H^2+\widetilde\kappa>0$ (see Hsiang and Hsiang~\cite{hshs1},
    Pedrosa and Ritor\'e~\cite{pri1}, and Abresch and
    Rosenberg~\cite{AbRo1}).  Let $\widetilde\kappa=\kappa-
    4\tau^2$. By the Daniel correspondence~\cite{da1}, there are
    (possibly nonembedded, see Torralbo~\cite{tor1}) spheres of
    constant mean curvature $H$ in $\E(\kappa,\tau)$ whenever
    $H^2+\tau^2=\widetilde{H}^2>\frac{-\widetilde \kappa}
    4=\frac{-\kappa} 4+\tau^2$.  }
\end{remark}

Hauswirth, Rosenberg and Spruck~\cite[Theorem~1.2]{hars1} used a
half-space type theorem for properly embedded $\frac{1}{2}$-surfaces
in $\H ^2\times \R $ to conclude that a complete, immersed
$\frac{1}{2}$-surface in $\H^2\times \R $ which is transverse
to $E_3$ must be an entire vertical graph.
% (i.e., it intersects exactly once every fiber of the Riemannian
% fibration $\pi \colon \E (-1,0)=\H^2\times \R \to \M ^2 (-1)=\H
% ^2$).
A bit later and using a different approach to prove a related
half-space type theorem for properly immersed minimal surfaces in the
Heisenberg space Nil$_3$, Hauswirth and Daniel~\cite[Theorem 3.1]{dh1}
demonstrated the corresponding result, namely that a complete immersed
 minimal
surface in Nil$_3$ which is transverse to $E_3$ must be an entire
vertical graph.
%These results allow us to improve Theorem~\ref{th:localgraph}
%in the cases of $\H^2\times \R $ and Nil$_3$.
Recently, Fern\'andez and Mira (personal communication) have
extended these results to the case of a complete vertical multigraph
$M$ in $\E (\kappa ,\tau )$ with mean curvature $H$ satisfying
$H^2=\frac{-\kappa }{4}$, concluding that $M$ is an entire vertical
graph. As a direct consequence, we can improve the statement of
item~2 in Theorem~\ref{th:localgraph}:
\begin{corollary}
\label{th:planes}
{\it Under the same hypotheses of Theorem~\ref{th:localgraph}, if
$H^2=\frac{-\kappa} 4$ then $M$ is either an entire vertical graph
or a vertical cylinder over a complete curve of geodesic curvature
$2H$ in $\M^2(\kappa)$. }
\end{corollary}
%\begin{corollary}\label{th:planes}
%\begin{description}
%\item[]
%\item[{\it i)}] Let $M$ be an orientable, parabolic, complete, immersed stable
%  $H$-surface in {\rm Nil}$_3$. Then, $H=0$ and either $M$ is
%  a vertical plane or an entire vertical graph.
%\item[{\it ii)}] Let $M$ be an orientable, parabolic, complete, immersed stable
%$H$-surface in $\H^2\times \R $. Then, $|H|\leq
%  \frac{1}{2}$ and if $|H|=\frac{1}{2}$, then either $M$ is a vertical
%  horocylinder (i.e. a vertical cylinder over a horocycle in $\H^2$)
%  or an entire vertical graph.
%\end{description}
%\end{corollary}
%\begin{proof}
%  We first prove item {\it i).}  Since Nil$_3=\E (\kappa ,\tau )$ with
%  $\kappa =0$, then Theorem~\ref{th:localgraph} insures that $H=0$ and
%  $M$ is either a vertical plane or a minimal vertical multigraph. In
%  this last case, the aforementioned result of Hauswirth and Daniel
%  gives that $M$ is an entire vertical graph. Finally, item~{\it ii)}
%  can be proved with similar arguments exchanging the result of
%  Hauswirth and Daniel by the one due to Hauswirth, Rosenberg and
%  Spruck.
%\end{proof}

\begin{remark}
  \label{remark3.4}
  {\rm Entire vertical minimal graphs in Nil$_3$ have been
    analytically classified by Fern\'andez and Mira~\cite{fm1} in
    terms of their conformal type ($\C $ or $\D =\{ z\in \C \mid
    |z|<1\} $) and a holomorphic quadratic differential (called
the {\it Abresch-Rosenberg quadratic differential} of the
graph~\cite{AbRo1}), whose only restriction is to be nonzero
when the conformal structure is $\C$.
    In spite of this analytic description, the geometry of the entire
    minimal graphs in Nil$_3$ is not well understood yet. All simple
    examples of such entire minimal graphs have area growth strictly
    greater than quadratic. We conjecture that no entire minimal graph
    in Nil$_3$ has quadratic area growth, which would imply that the
    only complete, immersed, stable $H$-surfaces in Nil$_3$ with
    quadratic area growth are vertical planes.

    Using the Daniel sister surface correspondence
    \cite[Corollary3.3]{da1} and the above paragraph, we conclude that
    there exist entire vertical graphs of constant mean curvature
    $\frac12$ in $\H^2\times \R $ which are parabolic.  It is natural to
    expect nonexistence of entire vertical graphs of constant mean
    curvature $\frac12$ and quadratic area growth in $\H^2\times\R$.
    Regarding the case $H\in [0,\frac{1}{2})$ in $\H ^2\times \R $,
    for each such a value of $H$ there exist complete stable
    $H$-surfaces of revolution in $\H^2\times \R $ which are entire
    graphs (see Nelli and Rosenberg~\cite{ner1}); even more, there
    exists an entire minimal graph which is conformally $\C $ (Collin
   and Rosenberg~\cite{cor2}). Note that the Daniel sister correspondence
relates these $H$-surfaces in $\H^2\times \R $, $0<H<\frac12$,
with $\widehat{H}$-surfaces in $\widetilde{\mbox{PSL}_2}(\R )=\E
(\kappa ,\tau )$ where $-1<\kappa <0$, $\tau ^2=\frac{1}{4}(\kappa
+1)$ and $\tau ^2+\widehat{H}^2=H^2<\frac{1}{4}$).

%  $\widetilde{H}\in \R $, because the expression $\tau _i^2+H_i^2$ is
%  independent of $i=1,2$ for two sister surfaces $M_i\subset \E
%  (\kappa _i,\tau _i)$, $i=1,2$; thus $\tau ^2+H^2= H^2<\frac{1}{4}$
%  for the Nelli-Rosenberg examples, while for an
%  $\widetilde{H}$-surface in Nil$_3=\E (0,\widetilde{\tau })$ with
%  $\widetilde{\tau }=\frac{1}{2}$ we have $\widetilde{\tau
%  }^2+\widetilde{H}^2\geq \frac{1}{4}$.
}
\end{remark}

We have already mentioned that if $M$ is an orientable, complete,
immersed, stable $H$-surface in $\E (\kappa ,\tau )$, then $H^2\leq
\frac{\tau ^2-\kappa }{3}$ (except the special case of
$M=\esf^2(\kappa )\times \{ t\} $ in $\esf^2(\kappa )\times \R $) by
a result due to Rosenberg~\cite{rose4}.  Next we will improve
slightly this inequality. We point out that
Corollary~\ref{th:mejora} below was stated in Remark~9.11
of~\cite{mpr19} in the particular case of $\E (\kappa ,\tau
)=\H^2\times \R $; although this particular case can be seen to
imply the whole statement by an application of the
Daniel sister correspondence, we will supply an independent proof.
We will need the following expression for the stability operator of
$M$ (see e.g.~\cite{rose4} or formula~(28) in~\cite{mpr19}):
\begin{equation}
\label{eq:q}
L=\Delta-K+\widetilde{q},\quad
\mbox{where }\widetilde{q}
=3 H^2+\kappa-\tau^2+(H^2-\mbox{det}(A)),
\end{equation}
where $K$ denotes the Gauss curvature of $M$.

Orientable, complete, immersed, stable $H$-surfaces $M$ in
$\E(\kappa,\tau)$, for $\kappa\geqslant \tau^2$, are classified
(this inequality leads to $\E (\kappa ,\tau )$ being
$\esf^2(\kappa )\times \R $ or a Berger sphere
with nonnegative scalar curvature; in the first case
$M$ must be a horizontal slice, while in the second case
there are no such complete stable $H$-surfaces for any value of $H$,
see Rosenberg~\cite{rose4} and Meeks, P\'erez and Ros~\cite{mpr19}).
Thus we can assume $\kappa<\tau^2$.

\begin{corollary}\label{th:mejora}
  Let $M$ be an orientable, complete, immersed, stable $H$-surface in
  $\E(\kappa,\tau)$, with $\kappa< \tau^2$. Then
%$H^2<\frac{\tau^2-\kappa} 3$. If additionally $M$ is embedded, then
there exists $\ve>0$ such that $H^2< \frac{\tau^2-\kappa} 3-\ve$.
\end{corollary}
\begin{proof}
We will first prove that the case $H^2=\frac{\tau ^2-\kappa }{3}$
cannot occur (recall that the case $H^2>\frac{\tau ^2-\kappa }{3}$
does not occur by~\cite{rose4}). Since
$H^2-\det(A)=\frac{1}{4}(k_1-k_2)^2\geq 0$ where
  $k_1,k_2$ are the principal curvatures of $M$ with respect to any
  unit normal vector field, then the function $\widetilde{q}$ in
  equation~(\ref{eq:q}) satisfies $\widetilde{q}\geq 0$.
  By~\cite[Theorem 2.9]{mpr19}), $M$ has quadratic area growth, in
  particular it is parabolic.
In this setting, Theorem~\ref{th:localgraph} implies that
$H^2\leq\frac{-\kappa} 4  $. Plugging the value of $H^2$
in this inequality we obtain $4\tau  ^2\leq \kappa $, which contradicts
our hypothesis. Therefore, $H^2<  \frac{\tau^2-\kappa }{3}$.

Next
%  Assume now that $M$ is embedded and we will prove the second
%  statement in the corollary. Again arguing by contradiction,
 suppose that for any $n\in\N$, there exists an orientable, complete,
%  embedded,
stable, $H_n$-surface $M_n$ in the same
  ambient space $\E (\kappa,\tau )$, with $H_n^2\in [0,
  \frac{\tau^2-\kappa} 3)$ converging to $H_{\infty
  }:=\frac{\tau^2-\kappa} 3$ as $n\to \infty $.  Similarly as in the
  proof of Theorem~\ref{th:localgraph}, we can assume after an ambient
  isometry that $M_n$ contains the origin ${\bf 0}\in \E (\kappa ,\tau
  )$. Since the sequence $\{ M_n\} _n$ has
  uniformly bounded second fundamental form (by stability), an
  accumulation point at the origin and their mean curvatures $H_n$
  satisfy $H_n\to H_{\infty }$, then after extracting a subsequence
  the $M_n$ converge to
% a weak $H_{\infty }$-lamination\footnote{See Definition~3.3 of~\cite{mpr19}
%  for the notion of a weak
%  $H$-lamination and see the two paragraphs just before Definition~3.7
%  of~\cite{mpr19} for a proof of this convergence result in the
%  similar case of taking limits around a limit point of a complete
%  embedded $H$-surface; the arguments there can be adapted to our
%  situation a straightforward manner.}
%of $\E (\kappa ,\tau )$ one whose leaves,
%   Note that, except for the
%  minimal case $H_{\infty }=0$, the surface $M_{\infty }$ might fail
%  to be embedded, and at every self-intersection of $M_{\infty }$
%  exactly two local leaves of $M_{\infty }$ do intersect, with
%  opposite (nonzero) mean curvature vectors.  Finally, the surface
%  $M_\infty$ is complete (because it is a leaf of a weak lamination of the
%  whole space $\E (\kappa ,\tau )$), and it is stable (as it is a limit
%  of the stable surfaces $M_n$).
to an orientable, complete, stable $H_{\infty }$-surface
$M_{\infty}\subset \E (\kappa ,\tau )$ with ${\bf 0}\in M_{\infty
}$. This contradicts the arguments in the previous paragraph, and
finishes the proof.
\end{proof}

\section{A Bernstein-type theorem for horizontal minimal graphs in {\rm Nil}$_3$.}
\label{sechorizgraphNil}
Consider the model of the Heisenberg space Nil$_3=\E (0,\frac{1}{2})$ given by $(\R^3,ds^2)$ with the
Riemannian metric
\begin{equation}
\label{eq:metricNil}
ds^2=dx^2+dy^2+\left( dz+\frac{1}{2}(y\, dx-x\, dy)\right) ^2.
\end{equation}
In this model, the Riemannian fibration is the vertical projection $\pi \colon \R^3\to \R^2$, $\pi
(x,y,z)=(x,y)$ and the vector fields
\begin{equation}
\label{eq:Ei}
E_1=\partial _x-\frac {y}{2}\partial _z,
\quad
E_2=\partial _y+\frac {x}{2}\partial _z,
\quad
E_3=\partial _z
\end{equation}
define a global orthonormal basis of left invariant vector fields.
$E_3$ is a (unit) Killing vector field of Nil$_3$, although this property
no longer holds for $E_1,E_2$.  Instead,
\[
X:=E_1+yE_3%\quad\mbox{and}\quad Y:=E_1-xE_3
\]
is Killing (not bounded), and its associated 1-parameter group of
isometries is given by
\[
\phi_t(x,y,z)=\left(x+t,y,z+\frac{ty} 2 \right),\quad t\in\R.
\]
Given a domain $\Omega\subset\R^2\equiv \{ (0,y,z)\mid
y,z\in \R \} $ and a $C^2$-function $u\colon \Omega\to \R$,
the {\it horizontal graph} $\Sigma_u$  defined by $u$
(in the direction of~$X$) is the surface of
Nil$_3$ parameterized by
\begin{equation}
  \label{eq:Fgrafo}
  F(y,z)=\left(u(y,z),y,z+\frac y 2 u(y,z)\right),\quad
  (y,z)\in\Omega,
\end{equation}
see Figure~\ref{fig1}.
\begin{figure}
\begin{center}
\includegraphics[height=5cm]{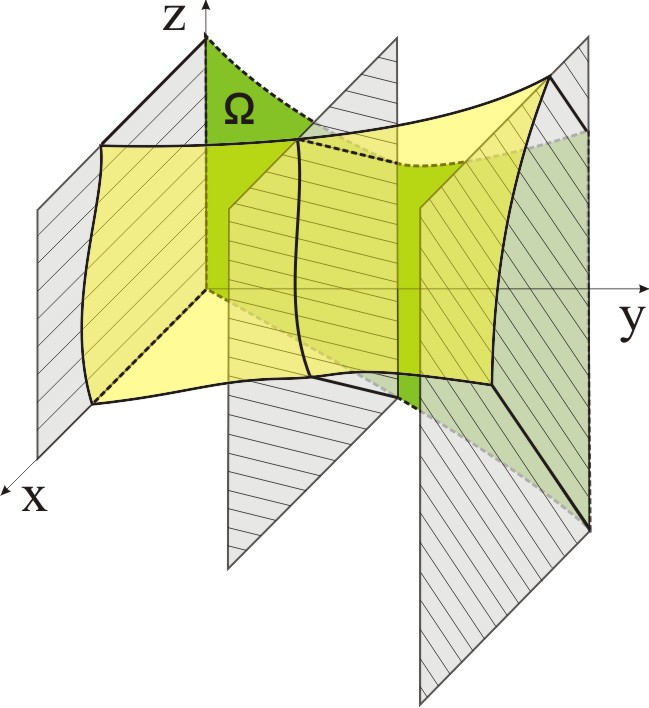}
\caption{A horizontal graph in the direction of $X$, over a domain
$\Omega $ of the $(y,z)$-plane. The parallel lines in vertical
planes are integral curves of $X$.} \label{fig1}
\end{center}
\end{figure}
Note that since the rotations around the $x_3$-axis are isometries
of Nil$_3$, this notion of horizontal graph does not really give
priority to the special horizontal direction defined by $X$.  A
straightforward computation gives that a horizontal graph $\Sigma_u$
is minimal if and only if $u$ satisfies the following PDE:
\[
\left[1+2y u_z+(1+y^2)u_z^2\right]u_{yy}
-2 u_y\left[y+(1+y^2)u_z\right]u_{yz}
\]
\[
+\left[1+(1+y^2)u_y^2\right]u_{zz}
-u_y u_z (1+yu_z)=0.
\]
A horizontal graph $\Sigma_u$ is said to be {\it entire} when
$\Omega=\R^2$, i.e. $\Sigma _u$ intersects exactly once every
integral curve of $X$. As simple examples of entire horizontal
minimal graphs, we have the following ones:
\begin{enumerate}
\item For any $a,b\in\R$ the vertical plane $\Pi_{a,b}=\{x=ay+b\}$ is
  the entire horizontal minimal graph $\Sigma _u$ associated to the
  function $u(y,z)=ay+b$. Moreover, the induced metric on $\Pi _{a,b}$
  is flat.
\item Given $c,d\in\R$, $c\neq 0$, the function $u(y,z)=cz+d$
  describes an entire horizontal minimal graph~$\Sigma_u$, which fails
  to be parabolic by Theorem~\ref{thm4.1} below.  $\Sigma _u$ can also
  be seen as an entire vertical minimal graph parameterized by
  \[
  \widetilde{F}(x,y)=\left(x,y,\frac{xy} 2+\frac{x-d}{c}\right),\quad
  (x,y)\in\R^2\equiv \{ (x,y,0)\mid x,y\in \R \} .
  \]
  % (to check this last property, simply recall that the equation of
  % the vertical minimal graphs $(x,y,f(x,y))$ is given by
  % $(1+\b^2)f_{xx}-2\al\be f_{xy}+(1+\aL^2)f_{yy}=0$, where
  % $\aL=f_x+\frac y 2$ and $\b=f_y-\frac x 2$).
\end{enumerate}
% Las líneas de flujo asociadas a $X$ no son geodésicas

If $M\subset \mbox{Nil}_3$ is a horizontal minimal graph and $\eta $
denotes a unit normal vector field to $M$, then $\langle X,\eta \rangle$
is a nonvanishing Jacobi function on $M$. Hence, Lemma~\ref{lemaFC}
insures that $M$ is stable. In fact, $M$ is area-minimizing by the
standard calibration argument: Suppose that $u\colon \Omega \to \R $
is a function defining a horizontal minimal graph over a domain
$\Omega \subset \{ (0,y,z)\mid (y,z)\in \R^2\} $. Consider the
differential $2$-form $\omega$ defined on the ``horizontal
cylinder'' $\cup _{t\in   \R } \phi _t(\Omega )$ by
\[
 \omega _{\phi _t(p)}(v_1,v_2)=dV_{F(p)}\left(
 \left( d\phi _{u(p)-t}\right) _{\phi _t(p)}(v_1),
 \left( d\phi _{u(p)-t}\right) _{\phi _t(p)}(v_2),\eta _{F(p)}
 \right),
\]
where $p\in \Omega $, $t\in \R $, $v_1,v_2\in
 T_{\phi _t(p)}\mbox{Nil}_3$, $dV$ is the volume form for the metric
 $ds^2$ on Nil$_3$, $F$ is given by (\ref{eq:Fgrafo}) and
 $\{ \phi _t\} _{t\in \R } $ is the one-parameter group generated by $X$
(which consists of isometries of Nil$_3$). Then $\omega $ is closed
(since $M$ is minimal), has unit comass and $M$ is a calibrated
surface for $\omega $, i.e. $\omega |_M$ equals the area element of
$M$. By the Fundamental Theorem of Calibrations (see e.g.
Morgan~\cite{mo2} Theorem~6.4, also see Harvey and
Lawson~\cite{hl2}), each compact subdomain of the graph of $u$ is
area minimizing for its boundary.

A natural version of the Bernstein problem in Nil$_3$ is to consider
entire horizontal minimal graphs (or more generally, with constant
mean curvature $H\in \R $) in the sense above.  As we said in
Remark~\ref{remark3.4}, the Bernstein problem for the minimal case
in the vertical direction (i.e. when the minimal surface is an
entire vertical graph) has been solved by Fern\'andez and
Mira~\cite{fm1} and the solutions of this problem are essentially
described in terms of holomorphic quadratic differentials on $\C $
(not identically zero) or on $\D $. Regarding the case of nonzero
constant mean curvature, Figueroa, Mercuri and Pedrosa~\cite[Theorem
4]{FMP} proved that there are no entire vertical graphs of nonzero
constant mean curvature in Nil$_3$. Their argument, based on the
maximum principle applied to spheres with the appropriate value of
the mean curvature (see the third paragraph in the proof of
Theorem~\ref{th:localgraph} for a similar argument), works without
changes when exchanging vertical graphs by horizontal ones. Hence,
we only have to concentrate on the minimal case of the horizontal
Bernstein problem. Theorem~\ref{thm4.1} below characterizes the
examples $\Pi_{a,b}$ above among solutions of the horizontal
Bernstein problem with parabolic conformal structure.

An orientable surface $M$ of Nil$_3$ with unit normal vector field $\eta $
is said to be a {\it horizontal multigraph} when $\langle X,\eta \rangle$
does not vanish. By Lemma~\ref{lemaFC}, every horizontal multigraph is
stable.

\begin{theorem}
  \label{thm4.1}
Let $M\subset \mbox{\rm Nil}_3$ be a complete minimal horizontal
multigraph. If $M$ is parabolic, then it is a vertical plane (which
is entire).
\end{theorem}

\begin{proof}
  Suppose $M$ is not a vertical plane. By Corollary~\ref{th:planes},
  $M$ is an entire vertical graph.
  Consider the Jacobi functions $u=\langle \eta ,E_3\rangle $, $v=\langle
  \eta ,X\rangle $, where $\eta $ is a unit normal vector field on~$M$.  Since
  $M$ is a horizontal multigraph, up to a change of orientation we can
  assume $v>0$.  Hence Theorem~\ref{thm1} gives that $u=\l v$ for some
  $\l \in \R -\{ 0\} $; this is, $\langle \eta ,\l X-E_3\rangle =0$ on
  $M$. Then $\l X-E_3$ restricts to a tangent vector field along $M$,
  which has no zeros since $\l \neq 0$ and $E_1,E_3$ are linearly
  independent.  As $\l X-E_3$ is Killing, then $M$ is invariant under
  the one-parameter group of isometries generated by $\l X-E_3$. By
  the results in Figueroa, Mercuri and Pedrosa~\cite{FMP}, $M$ must be
  ambiently isometric to an entire minimal vertical graph $M_\te$
  defined by
  \begin{equation}\label{eq:fmp}
    z=z(x,y)=\frac{xy}{2}+\frac{\sinh(2\te)}{2}\left[y\sqrt{1+y^2} +\ln
      \left( y+\sqrt{1+y^2}\right) \right] ,\quad (x,y)\in \R^2,
  \end{equation}
  for some $\te\in \R $. Since all isometries of Nil$_3$ preserve both
 the vertical direction and the horizontal distribution generated by
 $E_1,E_2$ (see equation (\ref{eq:Ei})), then it suffices to show that
for all $\te \in \R $, the surface $M_{\theta}$ is not a horizontal
multigraph in any horizontal direction, i.e. the function $\langle
X_{\al },\eta \rangle $ has a zero on $M_{\theta}$ for every $\al \in
[0,2\pi )$, where $X_{\al }$ is the horizontal Killing vector field
of Nil$_3$ given by
  \[
  X_\al=\cos\al\, (E_1+yE_3)+\sin\al\, (E_2-x E_3).
  \]
The vector fields
  \[
  T_1=E_1+y E_3,\qquad
  T_2= E_2+\sinh(2\te) \sqrt{1+y^2} E_3
  \]
  restrict to $M_\te$ as a global basis of the tangent bundle of
  $M_\te$, and
  \[
  \det(T_1,T_2,X_\al)=-\sin\al\left(x+\sinh(2\te)\sqrt{1+y^2}\right) .
  \]
  Therefore, $X_\al$ is tangent to $M_\te$ (hence $\langle X_{\al },\eta \rangle
  $ vanishes) along the curve with
  $x=-\sinh(2\te)\sqrt{1+y^2}$, for every $\al\in[0,2\pi)$.
\end{proof}

\medskip

\section{Appendix: Vertical cylinders in
  $\E(\kappa,\tau)$.}
\label{secappendix}
Consider the Riemannian submersion $\pi:\E(\kappa,\tau)\to
\M^2(\kappa)$. The vertical fibers of this submersion are geodesics,
and they are also integral curves of the unit Killing
vector field $E_3$, which generates the kernel of $d\pi $.
$\E (\kappa ,\tau )$ is parallelizable: there exists a global
orthonormal frame $(E_1,E_2,E_3)$ of the tangent bundle, and
\[
[E_1,E_2]=2\tau E_3,\quad
[E_2,E_3]=\sigma E_1,\quad
[E_3,E_1]=\sigma E_2
\]
where
\[
\sigma =\left\{ \begin{array}{ll}
0 & \mbox{if }\tau =0,
\\
\frac{\kappa }{2\tau } & \mbox{if }\tau \neq 0.
\end{array}\right.
\]
In particular, the 2-dimensional horizontal distribution
Span$(E_1,E_2)=\ker(d\pi )^{\perp }$ is integrable if $\tau =0$ (in
this case $\E (\kappa ,\tau )=\M ^2(\kappa )\times \R $) and it is
completely nonintegrable otherwise.  Denote by $\overline{\nabla}$
the Riemannian connection on $\E(\kappa,\tau)$. The nonvanishing
Christoffel symbols $\overline\Gamma_{ij}^k= \langle\overline{\nabla}_{E_i}
E_j,E_k\rangle$ are the following:
%\begin{equation}
%\label{eq:christoffel}
\[
\overline{\Gamma}_{12}^3= \overline{\Gamma}_{23}^1=
-\overline{\Gamma}_{21}^3=-\overline{\Gamma}_{13}^2=\tau\, , \qquad
\overline{\Gamma}_{32}^1=-\overline{\Gamma}_{31}^2=
\tau-\sigma
\]
and
\begin{equation}
\label{eq:1}
\left.
\begin{array}{l}
  \overline{\nabla}_{E_i} E_i=0, \qquad i=1,2,3, \\
  \overline{\nabla}_{E_1} E_2=-\overline{\nabla}_{E_2} E_1=\tau E_3,\\
  \overline{\nabla}_X E_3=\tau X\times E_3,
  %\\
%  \overline{\nabla}_{E_3}X=(\tau-\sigma) X\times E_3\\
%  \mbox{\bf ¿Por qué habíais
%    quitado la última línea?Simplifica el final de la demostración del lema
%  5.1.}
\end{array}
\right\}
\end{equation}
for any vector field $X$ on $\E(\kappa,\tau)$, where $\times $ denotes
the vector product in $\E (\kappa ,\tau )$ with respect to the direct
orthonormal frame $(E_1,E_2,E_3)$.

Given a curve $\g\subset \M^2(\kappa )$, we define the vertical
cylinder over $\g$ as $C_\g=\pi^{-1}(\g)$.  If we parameterize~$\g $
by its arc-length, then a global orthonormal frame of the tangent
bundle of $C_\g$ is given by $\{ \g', E_3\}$ (here we are identifying
$\g '$ with its isometric preimage under $d\pi $).  Let $\eta $ be the
unit normal vector field along $C_\g$ such that $\{\g', E_3, \eta \}$ is
positively oriented.

\begin{lemma}\label{lem:cylinders}
  The mean curvature with respect to $\eta $, the Gauss curvature and the
  extrinsic curvature of $C_{\g }$ are respectively given by
    \[
    H=\frac{k_\g}{2},\qquad K=0,\qquad K_{ext}=-\tau^2,
    \]
    where $k_\g$ is the geodesic curvature of $\g$ in $\M^2(\kappa)$.
\end{lemma}
\begin{proof}
  A straightforward computation using (\ref{eq:1}) gives that the
  matrix of the second fundamental form of $C_\g$ with
  respect to the orthonormal basis $\{\g',E_3\}$ of the tangent bundle
  to $C_\g$ is given by
  \begin{equation}
  \label{eq:II}
  \mbox{II}=\left(\begin{array}{cc} \langle \overline{\nabla}_{\g'}
      \g',\eta \rangle&
      \langle \overline{\nabla}_{\g'} E_3,\eta \rangle\\
      \langle \overline{\nabla}_{E_3}\g',\eta \rangle & \langle
      \overline{\nabla}_{E_3} E_3,\eta \rangle
    \end{array}\right)
  =\left(\begin{array}{cc}
      k_{\g} & \tau\\
      \tau & 0
    \end{array}\right)
   \end{equation}
  where $\langle\ ,\ \rangle$ denotes the Riemannian metric on $\E(\kappa,\tau)$.
In order to compute $K$, we use the Gauss equation:
\[
K=\overline{K}+\det (\mbox{II})=\overline{K}-\tau ^2,
\]
where $\overline{K}$ is the sectional curvature of the tangent plane to $C_{\g}$, i.e.
\[
\overline{K}=\langle \overline{\nabla }_{\g '}\overline{\nabla }_{E_3}E_3,\g '\rangle
-\langle \overline{\nabla }_{E_3}\overline{\nabla }_{\g '}E_3,\g '\rangle
-\langle \overline{\nabla }_{[\g ',E_3]}
E_3,\g '\rangle .
\]
Equations (\ref{eq:1}) and (\ref{eq:II}) imply that the first term
in the last expression vanishes and the second one equals $\tau ^2$.
It remains to check that the third one is zero and we will have that
$K=0$. Now, $[\g ',E_3]=\overline{\nabla }_{\g '}E_3-
\overline{\nabla }_{E_3}\g '$. Note that $\overline{\nabla }_{\g '}E_3=
\tau \eta $ is normal to $C_{\g }$ and that $\langle \overline{\nabla }_{E_3}\g ',
\g '\rangle =\langle \overline{\nabla }_{E_3}\g ',E_3\rangle =0$. Hence,
$[\g ',E_3]$ is normal to $C_{\g}$. Since
$[\g ',E_3]$ is clearly tangent to $C_{\g }$, it must be zero.
\end{proof}

\begin{proposition}\label{prop:cyl}
  Let $\g\subset \M^2(\kappa )$ be a complete curve with constant geodesic curvature
  $k_\g\in \R $.  Then the cylinder $C_\g$ is stable if and only if $\kappa\leq -k_\g^2$.
\end{proposition}
\begin{proof}
The stability operator of $C_\g$ is given by
$L=\Delta+|A|^2+\mbox{Ric}(\eta )$,  where $A$ is the
shape operator of $C_\g$.  By equation (\ref{eq:II}) we have $|A|^2=
  k_\g^2+2\tau^2$, while a direct calculation using (\ref{eq:1}) gives $\mbox{Ric}(\eta )=
  \kappa-2\tau^2$. Therefore, $L$ is just a translation
   (with constant potential) of the Laplacian on $C_{\gamma }$, namely
  \[
  L=\Delta+k_\g^2+\kappa.
  \]
Hence, the Dirichlet problem with zero boundary values for this
operator has first eigenvalue
$\l_1(\Omega)=\l_1^\Delta(\Omega)-(k_\g^2+\kappa)$ on every
relatively compact domain $\Omega\subset C_{\g }$, where
$\l_1^\Delta(\Omega)$ is the first eigenvalue of the Dirichlet
problem with zero boundary values for the Laplacian on $C_{\g }$.
Since the induced metric on $C_{\g }$ is flat, then
$\mbox{inf}_\Omega\ \l_1(\Omega)=-(k_\g^2+\kappa)<0$, from where the
proposition follows.
\end{proof}

\begin{remark}
 {\rm
  From Proposition~\ref{prop:cyl}, all constant
  mean curvature cylinders $C_\g\subset \E (\kappa ,\tau )$ are
  unstable when $\kappa>0$, and the only
  stable constant mean curvature cylinders
 in $\E (0,\tau )$ are vertical planes (which are minimal).
}
\end{remark}

\bibliographystyle{plain}
%\bibliography{bill.bib}

\end{document}